\documentclass[
  a4paper, 
  reqno, 
  oneside, 
  11pt
]{amsart}

\usepackage{vmargin}
\setpapersize{A4}
\setmargins
{2.5cm}		% margen izquierdo
{1.5cm}		% margen superior
{16cm}		% ancho del texto
{23.42cm}	% altura del texto
{10pt}		% altura de los encabezados
{1cm}		% espacio entre el texto y los encabezados
{0pt}		% altura del pie de pagina
{2cm}		% espacio entre el texto y el pie de pagina

\usepackage[utf8]{inputenc}
\usepackage{cancel}
\usepackage{multicol}
\usepackage[dvipsnames]{xcolor} % Must come before tikz!
\colorlet{cite}{LimeGreen!50!Green}
\usepackage{tikz}  % for diagrams
\usetikzlibrary{matrix,arrows,
decorations.markings,calc,intersections,shapes,backgrounds,positioning}
\tikzset{ 
  baseline=-2.3pt,
  text height=1.5ex, text depth=0.25ex,
  >=stealth,
  node distance=2cm,
  mid/.style={fill=white,inner sep=2.5pt},
}

\usepackage{lmodern}
\usepackage{tikz-cd}
\usepackage{amsthm, amssymb, amsfonts,amsmath,bbding, mathtools}

\usepackage{graphicx,caption,subcaption}
\usepackage{braket}  % For nice sets
\usepackage{microtype}
\usepackage{cancel}
\usepackage{parskip}
\usepackage[%
  bookmarks=true,			% show bookmarks bar?
  unicode=true,			% non-Latin characters in Acrobat’s bookmarks
  pdftitle={},		%
  pdfauthor={Edoardo Ballico}{Elizabeth Gasparim}{Francisco Rubilar}{Bruno Suzuki},	% 
  pdfkeywords={}{}{}{},	% list of keywords
  colorlinks=true,		% false: boxed links; true: colored links
  linkcolor=black,			% color of internal links
  citecolor=black,		% color of links to bibliography
  filecolor=black,		% color of file links
  urlcolor=black			% color of external links
]{hyperref}				% One of the most useful packages I have found.
\usepackage{cleveref}

\usepackage{enumitem}
\setenumerate[0]{label=(\textit{\roman*}\hair),font=\rm}

\theoremstyle{plain}	% 'plain' is the default.  The others are 'definition' and 'remark'.
\numberwithin{equation}{section}

\newtheorem{theorem}[equation]{Theorem}
\newtheorem*{theorem*}{Theorem}
\newtheorem{proposition}[equation]{Proposition}
\newtheorem{lemma}[equation]{Lemma}
\newtheorem*{lemma*}{Lemma}

\newtheorem*{corollary*}{Corollary}
\newtheorem*{remark*}{Remark}

\theoremstyle{definition}

\newtheorem{definition}[equation]{Definition}
\newtheorem{remark}[equation]{Remark}
\newtheorem{example}[equation]{Example}
\newtheorem{notation}[equation]{Notation}

\newcommand{\ce}{\mathrel{\mathop:}=}

\DeclareMathOperator{\Pic}{Pic}
\DeclareMathOperator{\Tot}{Tot}
\DeclareMathOperator{\Ext}{Ext}
\DeclareMathOperator{\rk}{rk}

\usepackage{eqparbox}
\usepackage{tikz-cd}
\usetikzlibrary{patterns, calc, arrows}

\newcommand{\trevo}{\textnormal{\tiny{\FourClowerSolid}}} %solid clover

\newcommand{\blank}{\mkern3mu\cdot\mkern3.7mu}
\newcommand{\llrr}[1]{%
 [\mkern-3.1mu[#1]\mkern-3.1mu]}
\newcommand{\Wedge}{\mathrm{\Lambda}\mkern-2mu}

\newcommand{\SC}{0.7}

\newcommand{\ccfan}{
\begin{minipage}{0.1\textwidth}
\centering \vspace{4pt}
\begin{tikzcd}
	\fill[pattern=north west lines, pattern color=SkyBlue] (0,0) rectangle (1*\SC,1*\SC);

	\draw[very thick,->] (0,0) -- (1*\SC,0);
	\draw[very thick,->] (0,0) -- (0,1*\SC);
	\begin{scope}
		\foreach \x in {0,1}
		{	\foreach \y in {0,1}
			{
			\filldraw (\x*\SC,\y*\SC) circle (0.03cm);
			}
		}
	\end{scope}	
\end{tikzcd}
\vspace{4pt}
\end{minipage}
}
\newcommand{\cfan}{
\begin{minipage}{0.05\textwidth}
\centering \vspace{4pt}
\begin{tikzcd}
	\draw[very thick,->] (0,0) -- (0,1*\SC);
	\begin{scope}
		\foreach \x in {0}
		{	\foreach \y in {0,1}
			{
			\filldraw (\x*\SC,\y*\SC) circle (0.03cm);
			}
		}
	\end{scope}	
\end{tikzcd}
\vspace{4pt}
\end{minipage}
}

\newcommand{\Zzero}{
\begin{minipage}{.2\textwidth}
\centering \vspace{4pt}
\begin{tikzcd}
	\fill[pattern=north west lines, pattern color=SkyBlue] (0,0) rectangle (-1*\SC,1*\SC);
	\fill[pattern=north east lines, pattern color=SkyBlue] (0,0) rectangle (1*\SC,1*\SC); 
	\draw[very thick,->] (0,0) -- (1*\SC,0);
	\draw[very thick,->] (0,0) -- (0,1*\SC);
	\draw[very thick,->] (0,0) -- (-1*\SC,0);
	\begin{scope}
		\foreach \x in {-1,...,1}
		{	\foreach \y in {0,1}
			{
			\filldraw (\x*\SC,\y*\SC) circle (0.03cm);
			}
		}
	\end{scope}	
\end{tikzcd}
\vspace{4pt}
\end{minipage}
}

\title[Quantization of Calabi--Yau threefolds]{Quantizations  of local Calabi--Yau threefolds \\
and their moduli of vector bundles}
\author{
E. Ballico, 
E. Gasparim,  
F. Rubilar, 
B. Suzuki}
\address{
Ballico - Dept. Mathematics, Univ. of Trento,  Povo  Italy; ballico@science.unitn.it, \newline
Gasparim -  Depto. Matem\'aticas, Univ. Cat\'olica del Norte,  Chile;  etgasparim@gmail.com, \newline 
Rubilar - Depto. Matem\'aticas, Univ. del Bío-Bío, Chile; francisco.rubilar.arriagada@gmail.com,\newline
Suzuki* - Depto. Matem\' atica, Univ. de S\~ao Paulo, Brazil; obrunosuzuki@gmail.com. \newline
*Corresponding author
}

\begin{document}

\begin{abstract}We describe the geometry of noncommutative deformations of local Calabi--Yau threefolds, 
showing that the choice of Poisson structure strongly influences the geometry of the quantum moduli space.
\end{abstract}
\maketitle



\tableofcontents

\section{Introduction}
We discuss moduli of vector bundles on those noncommutative local Calabi--Yau threefolds that occur 
in noncommutative crepant resolutions of the generalised conifolds $xy-z^nw^m=0$. 
Such crepant resolutions  require lines of type $(-1,-1)$ and $(-2,0)$, that is, those locally modelled by
 $$W_1\ce \Tot(\mathcal O_{\mathbb P^1}(-1)\oplus \mathcal O_{\mathbb P^1}(-1)) \quad 
 \text{or }\quad  W_2\ce \Tot(\mathcal O_{\mathbb P^1}(-2)\oplus \mathcal O_{\mathbb P^1}(0) ).$$ 
Their appearance is balanced in a precise sense described in \cite{GKMR} so that no particular 
configuration of such lines is more likely to occur in a crepant resolution than any other.

 Our results show that the structure of the quantum moduli space (Def.~\ref{ncmoduli}) of vector bundles over a noncommutative 
deformation varies drastically depending on the choice of a Poisson structure.

 In the 2-dimensional case, \cite{BG} described 
 the geometry of noncommutative deformations of the local surfaces $Z_k\ce  \Tot(\mathcal O_{\mathbb P^1}(-k))$, 
showing that the quantum moduli space of instantons over a noncommutative deformation $(\mathcal Z_k, \sigma)$ 
can be viewed as the étale space of a constructible sheaf over the classical moduli space of instantons on $Z_k$.
While in  2 dimensions vector bundles occur as mathematical representations of instantons, 
in the 3-dimensional case vector bundles occur as mathematical descriptions of BPS states, with $W_1$ and $W_2$ 
appearing as building blocks, as described in \cite{GKMR,GSTV, OSY}.

 In this work, we describe the geometry of    noncommutative deformations   $\mathcal W$ 
 of a Calabi--Yau threefold $W$, showing that the quantum moduli space of vector bundles on 
  together with the map taking a vector bundle on $\mathcal W$ to its classical limit
$$
\begin{tikzpicture}[baseline=-.4em]
\matrix (m) [matrix of math nodes, row sep=1.25em, inner sep=2pt,
column sep=0, ampersand replacement=\&]
{\mathfrak M_j^{\hbar}(\mathcal W, \sigma) \\ \mathfrak M_j( W) \\};
\path[-stealth,line width=.6pt,font=\scriptsize]
(m-1-1) edge (m-2-1)
;
\end{tikzpicture}
$$
has the structure of a constructible sheaf, whose rank and singularity set depend explicitly on the
choice of a noncommutative deformation.
In particular, we describe the geometry of noncommutative deformations
of some crepant resolutions. It is at this point yet unclear how these compare  with Van den Bergh's noncommutative crepant resolutions \cite{V}. Other approaches to  noncommutative  local Calabi--Yau threefolds  using the languages  of quivers and motives appeared in 
\cite{N,CMPS} with similar jump behaviour being observed,  and  applications to noncommutative gauge theory were 
described in \cite{S}.

To each Poisson structure $\sigma$ on $W_k$, with $k=1$ or $k=2$ there corresponds a noncommutative deformation 
$(\mathcal W_k, \mathcal A^\sigma)$ with $\mathcal A^\sigma = (\mathcal O_{W_k} \llrr{\hbar},\star_\sigma)$
where $\star_\sigma$ is the star product corresponding to $\sigma$. All of these Poisson structures were described in 
\cite{BGKS} in terms of generators over global functions; when $\sigma$ is one of such  generators, we refer to it as a 
{\it basic} Poisson structure.
 There exist Poisson structures for which
all brackets vanish on the first formal neighbourhood of $\mathbb P^1 \subset W_k$;  we call them {\it extremal}
Poisson structures, they behave very differently from the basic ones. Our main results are:

\begin{theorem*}[\ref{e1},\ref{e2}] Let $k=1$ or $2$.
If $\sigma$ is an extremal Poisson structure on $W_k$, then
 the quantum moduli space $\mathfrak M_j^{\hbar} (\mathcal W_k, \sigma)$ can 
be viewed as the étale space of a constructible sheaf  $\mathcal E_k$ of generic 
rank $2j-k-1$ over the classical moduli space 
$\mathfrak M_j (W_k)$ with singular stalks of  all  ranks up to $4j -k-4$.
\end{theorem*}

If  $\sigma'$ is another Poisson structure on $W_k$, then the corresponding sheaf $\mathcal E'_k$ is a 
subsheaf of $\mathcal E_k$, with the smallest possible sheaf occurring for basic Poisson structures.

\begin{theorem*}[\ref{iso1},\ref{iso2}] Let $k=1$ or $2$. If  $\sigma $ is a basic Poisson structure on $W_k$,
then the quantum moduli space $\mathfrak M^{\hbar}_j(\mathcal W_k, \sigma)$
and its classical limit  are isomorphic:
 $$\mathfrak M^{\hbar}_j(\mathcal W_k, \sigma)\simeq \mathfrak M_j(W_k)\simeq \mathbb P^{4j-5}.$$
 \end{theorem*}
 
 Therefore, comparing these results, we see that the choice of  Poisson structure
  has a strong influence  on the geometry of the quantum moduli space.

\section{Noncommutative deformations}\label{ncsec}

A {\it holomorphic Poisson structure}
on a complex manifold (or smooth complex algebraic variety) $X$ is given by a holomorphic bivector field
$\sigma \in H^0 (X, \Wedge^2 \mathcal T_X)$ whose Schouten--Nijenhuis bracket $[\sigma, \sigma] \in H^0 (X, \Wedge^3 \mathcal T_X)$ is zero. The associated Poisson bracket is then given by the pairing $\langle \blank {,} \blank \rangle$ between vector fields and forms
$
\{ f, g \}_\sigma = \langle \sigma, \mathrm d f \wedge \mathrm d g \rangle.
$

To obtain a noncommutative deformation of $X$ one must first promote the Poisson structure to a  {\it star product} on $X$, that   is, a $\mathbb C \llrr{\hbar}$-bilinear associative product
$
\star \colon \mathcal O_X \llrr{\hbar} \times \mathcal O_X \llrr{\hbar} \to \mathcal O_X \llrr{\hbar}
$
which is of the form
$
f \star g=  fg + \sum_{n=1}^\infty B_n (f,g) \, \hbar^n$
where the $B_n$ are bidifferential operators.

The   pair $(X, \star_\sigma)$ is called a 
{\it deformation quantization} of $(X, \sigma)$ when the star product on $X$ satisfies $B_1 (f, g) = \{ f, g \}_\sigma$.

For a holomorphic Poisson manifold $(X, \sigma)$  with associated Poisson bracket $\{ \blank {,} \blank \}_\sigma$, 
 the {\it sheaf of formal functions with holomorphic coefficients on the quantization $(X, \star_\sigma)$} 
 is 
 $$\mathcal A^\sigma := (\mathcal O_X \llrr{\hbar},\star_\sigma).$$

When we work with a fixed Poisson structure, we use the abbreviated notations $\mathcal A$, $\{ \blank {,} \blank \}$ and $\star$. We also use the cut to order $n$ represented as $\mathcal A^{(n)} = \mathcal O_X[[\hbar]]/ \hbar^{n+1}$.

The existence of star products on Poisson manifolds was  proven in the seminal papers
of Kontsevich \cite{kontsevich1, kontsevich2}. 
For  a complex algebraic variety $X$ with structure sheaf $\mathcal{O}_X$, suppose
\begin{equation}\label{coh}
H^1 (X, \mathcal{O}_X) =H^2 (X, \mathcal{O}_X) = 0,
\end{equation}
then 
\[
\{ \text{\rm Poisson deformations of $\mathcal O_X$} \} / {\sim} \; \leftrightarrow \; \{ \text{\rm associative deformations of $\mathcal O_X$} \} / {\sim}
\]
where $\sim$ denotes gauge equivalence \cite[Cor.\thinspace 11.2]{Y}. 

We will use the following result.
\begin{lemma}\label{cone} \cite[Prop.\thinspace 1]{BGKS} 
Let   $X$ be a smooth complex threefold and $\sigma$ a Poisson structure on $X$, 
then $f\sigma$ is integrable for all $f\in \mathcal O(X)$. 
\end{lemma}

From now on we will focus on the local Calabi--Yau threefolds $W_k$ defined as follows. 
For $k \geq 1$, 
\[
W_k = \Tot (\mathcal{O}_{\mathbb{P}^1}(-k) \oplus \mathcal{O}_{\mathbb{P}^1}(k-2)).
\]
The {\it canonical charts} for the complex manifold structure of $W_k$  is obtained by gluing the open sets 
$$U = \mathbb{C}^3_{\{z,u_1,u_2\}}  \quad \mbox{and} \quad  V = \mathbb{C}^3_{\{\xi,v_1,v_2\}}$$ 
by the relation
$$(\xi, v_1, v_2) = (z^{-1}, z^k u_1, z^{-k+2}u_2).$$

If $\sigma$ is a Poisson structure on $W_k$, we call $\mathcal W_k (\sigma) = (W_k, \mathcal A^\sigma)$ a {\it noncommutative deformation} of $W_k$,
and a {\it vector bundle} on a noncommutative deformation is by definition a locally free sheaf of $\mathcal A^\sigma$-modules. 
These vector bundles and their moduli are our objects of study here.

The cohomological hypothesis \eqref{coh} is verified in the cases of $W_1$ and $W_2$ (but not for $W_3$, see App.\thinspace\ref{H1-Wk-O}).
We now recall the basic properties of Poisson structures on $W_k$ for $k=1,2.$ 
All Poisson structures on $W_k$ may be described  by giving their generators over global functions.

All Poisson structures on $W_1$ can be obtained using the following generators
\cite[Thm.\thinspace 3.2]{BGKS}
$$
\sigma_1= \partial z \wedge \partial u_1, \quad
\sigma_2= \partial z \wedge \partial u_2,
$$
$$
\sigma_3= u_1 \partial u_1 \wedge \partial u_2
- z \partial z \wedge \partial u_2, \quad
\sigma_4= u_2 \partial u_1 \wedge \partial u_2
+z \partial z \wedge \partial u_1.
$$
The $W_1$-Poisson structures  $\sigma_1,\sigma_2, \sigma_3,\sigma_4$   are pairwise isomorphic.

All Poisson structures on $W_2$ can be obtained using the following generators \cite[Lem.\thinspace3]{BGKS}
$$
\sigma_1 = \partial z \wedge \partial u_1, \quad 
\sigma_2= \partial z \wedge \partial u_2, \quad
\sigma_3 = z \partial z \wedge \partial u_2, $$
$$
\sigma_4 = u_1 \partial u_1 \wedge \partial u_2, \quad
\sigma_5 = 2zu_1 \partial u_1 \wedge \partial u_2
	- z^2 \partial z \wedge \partial u_2.
$$
 The $W_2$-Poisson structures $ \sigma_2$ and $\sigma_5$   are isomorphic.
Moreover, the Poisson structures  $\sigma_1,\sigma_2, \sigma_3,\sigma_4$ on $W_2$  are pairwise inequivalent, 
giving 4 distinct Poisson manifolds.

\section{Vector  bundles on noncommutative deformations}

To discuss moduli of vector bundles on noncommutative deformations of $W_k$, for $k=1$ or $2$
we will consider those bundles that are formally algebraic. 
We first present some generalities.

\begin{definition}\label{formally-algebraic}
We say that ${\bf p} = \sum p_n \hbar^n \in \mathcal O_X {[[\hbar]]}$ is formally algebraic if $p_n$ is a polynomial for every $n$.
We say that a vector bundle is {\it formally algebraic} if it is isomorphic to a vector bundle given by formally algebraic transition functions. In addition, if there exists $N$ such that $p_n=0$ for all $n>N$, we then say that ${\bf p}$ is {\it algebraic}.
\end{definition}

\begin{lemma}
\label{acyclic}
Let $\mathcal A$ be a deformation quantization of $\mathcal O_X$. Then an $\mathcal A$-module $\mathcal S$ is acyclic if and only if $S = \mathcal S / \hbar \mathcal S$  is acyclic.
\end{lemma}

\begin{proof}
Consider the short exact sequence 
\[
0 \longrightarrow \mathcal S \stackrel{\hbar}\longrightarrow \mathcal S \longrightarrow S \longrightarrow 0.
\]
It gives, for $j > 0$ surjections 
\[
H^j (X, \mathcal S) \stackrel{\hbar}\longrightarrow H^j (X, \mathcal S) \longrightarrow 0.
\]
This immediately implies that $H^j (X, \mathcal S) = 0$ for $j > 0$. The converse is immediate.
\end{proof}

\begin{notation}
Let $\mathcal{W}_k$ be a noncommutative deformation of $W_k$. Denote by $\mathcal{A}(j)$ the line bundle over 
$\mathcal {W}_k$ with transition function $z^{-j}$, hence the pullback of $\mathcal O_{\mathbb P^1}(j)$.
\end{notation}

\begin{proposition}\label{lineb}
For $k=1,2$ any line bundle on $\mathcal{W}_k$ is isomorphic to $\mathcal{A}(j)$ for some $j \in \mathbb{Z}$, i.e., $\Pic(\mathcal W_k) = \mathbb{Z}$ when $k=1,2$.
\end{proposition}
\begin{proof}
Let $f = f_0 + \sum_{n=1}^\infty \widetilde f_n \,\hbar^n \in \mathcal A^* (U \cap V)$ be the transition function for 
the line bundle $\mathcal L$. Then there exist functions $a_0 \in \mathcal O^* (U)$ and $\alpha_0 \in \mathcal O^* (V)$ such that $\alpha_0 f_0 a_0 = z^{-j}$ and viewing $a_0$ resp.\ $\alpha_0$ as elements in $\mathcal A^* (U)$ resp.\ $\mathcal A^* (V)$ one has $\alpha_0 \star f \star a_0 = z^{-j} + \sum_{n=1}^\infty f_n \hbar^n$ for some $f_n \in \mathcal O (U \cap V)$. We may thus assume that the transition function of $\mathcal L$ is 
$z^{-j} + \sum_{n=1}^\infty f_n \hbar^n.$

To give an isomorphism $\mathcal L \simeq \mathcal A (j)$ it suffices to define functions $a_n \in \mathcal O (U)$ and $\alpha_n \in \mathcal O (V)$ satisfying
\begin{align}
\label{linebundleisomorphism}
\big( 1 + \textstyle\sum_{n=1}^\infty \alpha_n \hbar^n \big) \star \big( z^{-j} + \sum_{n=1}^\infty f_n \hbar^n \big) \star \big( 1 + \sum_{n=1}^\infty a_n \hbar^n \big) = z^{-j}.
\end{align}

Collecting terms by powers of $\hbar$, (\ref{linebundleisomorphism}) is equivalent to the system of equations
\begin{align}
S_n + z^{-j} a_n + z^{-j} \alpha_n = 0 \tag*{$n = 1, 2, \dotsc$}
\end{align}
where $S_n$ is a finite sum involving $f_i$, $B_i$ for $i \leq n$, but only $a_i, \alpha_i$ for $i < n$. The first terms are
\begin{align*}
S_1 &= f_1 \\
S_2 &= f_2 + \alpha_1 f_1 + a_1 f_1 + B_1 \big( \alpha_1, z^{-j} \big) + B_1 \big( z^{-j}, a_1 \big) + \alpha_1 z^{-j} a_1 \\
S_3 &= f_3
+ B_2 \big( \alpha_1, z^{-j} \big)
+ B_2 \big( z^{-j}, a_1 \big)
+ B_1 \big( \alpha_2, z^{-j} \big) \\ &\qquad
+ B_1 \big( z^{-j}, a_2 \big)
+ B_1 \big( \alpha_1, f_1 \big)
+ B_1 \big( \alpha_1, z^{-j} a_1 \big)
+ B_1 \big( z^{-j}, a_1 \big) \\ &\qquad
+ \alpha_2 f_1 + \alpha_2 z^{-j} a_1 + \alpha_1 f_2 + \alpha_1 f_1 a_1 + \alpha_1 z^{-j} a_2 + f_2 a_1 + f_1 a_2
\end{align*}
Since by Lem.~\ref{van}  we have $H^1 (W_k, \mathcal O_{W_k}) = 0$ when $k=1,2$, we can solve these equations recursively, by defining $a_n$ to cancel out all terms of $z^j S_n$ having positive powers of $z$ and setting $\alpha_n = z^j S_n - a_n$.
\end{proof}

Note that this is essentially the same proof as \cite[Prop.\thinspace 6.7]{BG},
and it does not work for $k\geq 3$; in fact,  $\Pic(W_3)$ is much larger,
see Lem.\thinspace\ref{muchlarger}.

We now consider vector bundles of higher rank.

\begin{theorem}
\label{ncfiltrable}
 For $k=1,2$, vector bundles over $\mathcal W_k (\sigma)$ are filtrable.
\end{theorem}
\begin{proof}
This is a generalisation of Ballico--Gasparim--K\"oppe \cite[Thm.~3.2]{ballicogasparimkoppe1} to the noncommutative case. Let  
$\mathcal E$ be a sheaf of $\mathcal A$-modules. 
Lem.~\ref{acyclic} gives that the classical limit
$\mathcal E_0 = \mathcal E / \hbar \mathcal E$ is acyclic as a sheaf of $\mathcal A$-modules (and equivalently as a sheaf of $\mathcal O$-modules) if and only if $\mathcal E$ is acyclic as a sheaf of $\mathcal A$-modules. 
 
Filtrability for a bundle $E$ over $W_k$, for $k=1,2$ was  proved in \cite{K} and 
 is obtained from the vanishing of cohomology groups $H^i (W_k, E \otimes {\mbox{Sym}^n} N^*)$ for $i=1,2$, where $N^*$ is the conormal bundle of $\ell \subset W_k$ and $n>0$ are integers; the proof proceeds by induction on $n$. In the noncommutative case, let $\mathcal S$ denote the kernel of the projection $\mathcal A^{(n)} \to \mathcal A^{(n-1)}$. By construction we have that $\mathcal S / \hbar \mathcal S = {\mbox{Sym}^n} N^*$ and the required vanishing of cohomologies is guaranteed by Lem.\thinspace\ref{acyclic}. \end{proof}

The analogous proof does not work for $W_3$, see \cite[Rem.\thinspace 3.13]{K}. It is unknown whether 
bundles on $W_k$ are filtrable when $k\geq 3$.

\begin{remark}
There are also some particular features happening only when $k=1$. 
Every holomorphic vector bundle on $W_1$ is  algebraic  \cite[Thm.\thinspace 3.10]{K}, and 
$W_1$ is formally rigid \cite[Thm.\thinspace 11]{GKRS}.
In contrast, if $k>1$, then $W_k$ has  as infinite-dimensional family of deformations. 
In particular, a deformation family for $W_2$ can be given by
$
(\xi, v_1, v_2) = \left(z^{-1}, z^2u_1 + z \sum_{j>0} t_ju_2^j, u_2\right)
$ 
\cite[Thm.\thinspace 13]{GKRS} and this family 
contains infinitely many distinct manifolds \cite[Thm.\thinspace 1.13]{BGS}. 
Furthermore, for $k > q > 0$, $W_k$ can be deformed to $W_q$ \cite[Thm.\thinspace 1.28]{BGS}.  
\end{remark}

For each  Poisson  manifold $(W_k, \sigma)$, we want to study moduli spaces of vector bundles over $(\mathcal W_k, \sigma_{\star})$
where $\sigma_{\star}$ is the corresponding star product.

\cite[Prop.\thinspace 3.1]{K} showed that a rank $2$ bundle $E$ on $W_k$ 
with first Chern class $c_1 (E) = 0$  is determined by a  canonical transition matrix
$	\begin{pmatrix}
		z^j & p \\
		0  & z^{-j}
	\end{pmatrix}$
where, using  $\epsilon = 0,1$ we have: 
\begin{equation}\label{poli1}p=\sum_{s=\epsilon}^{2j-2}\sum_{i=1-\epsilon}^{2j-2-s}\sum_{l=i+s-j+1}^{j-1}p_{lis}z^lu_1^iu_2^s\quad \text{for} \quad k=1,\end{equation} and
\begin{equation}\label{poli2}p=\sum_{s=\epsilon}^{\infty}\sum_{i=1-\epsilon}^{j-1}\sum_{l=2i-j+1}^{j-1}p_{lis}z^lu_1^iu_2^s\quad \text{for} \quad  k=2.\end{equation}

Accordingly, for a noncommutative deformation $(\mathcal W_k,\sigma)$ we define the 
notion of {\it canonical  transition matrix} as:
\begin{equation}\label{ncpoly}T =
\begin{pmatrix}
z^j & {\bf p} \\
 0  & z^{-j}
\end{pmatrix}
\qquad \mbox{with} \qquad
 {\bf p} = \sum_{n=0}^\infty p_n \hbar^n 
 \in \Ext^1 (\mathcal A (j), \mathcal A (-j)).
\end{equation}

Where we have that each $p_n$ can be given the same canonical form of the classical case, as shown by the next result, which is valid for $k=1$ or $2$. 
%We set $\mathcal{O} \ce \mathcal{O}_{W_k}$.

\begin{lemma}
\label{subspace}
Let $\mathcal A$ be a deformation quantization of $\mathcal{O} \ce \mathcal O_{W_k}$. There is an injective map of $\mathbb C$-vector spaces
\begin{align*}
\begin{tikzpicture}[baseline=-2.6pt,description/.style={fill=white,inner sep=2pt}]
\matrix (m) [matrix of math nodes, row sep=1em, text height=1.5ex, column sep=1.5em, text depth=0.25ex, ampersand replacement=\&, column 2/.style={anchor=base west}]
{
\Ext^1_{\mathcal A} (\mathcal A (j), \mathcal A (-j)) \& \displaystyle\prod_{n=0}^\infty \Ext^1_{\mathcal O} (\mathcal O (j), \mathcal O (-j)) \hbar^n \simeq \Ext^1_{\mathcal O} (\mathcal O (j), \mathcal O (-j)) \llrr{\hbar} \\
 {\bf p} = p_0 + \displaystyle\sum_{n=1}^\infty p_n \hbar^n \& (p_0, p_1 \hbar, p_2 \hbar^2, \dotsc) \\
}
;
\path[|-stealth,line width=.5pt,font=\scriptsize]
(m-2-1) edge (m-2-2)
;
\path[-stealth,line width=.5pt,font=\scriptsize]
(m-1-1) edge (m-1-2)
;
\end{tikzpicture}
\end{align*}
where $p_i \in \Ext^1 (\mathcal O (j), \mathcal O (-j))$.
\end{lemma}

\begin{proof}
$\Ext^1_{\mathcal A} (\mathcal A (j), \mathcal A (-j)) $ is the quotient of $\Ext^1_{\mathcal O} (\mathcal O (j), \mathcal O (-j)) \llrr{\hbar}$
by the relations\\ $q_n \simeq q_n + \sum p_i p_{n-i}.$\end{proof}

We  wish to describe the structure of moduli spaces 
of vector bundles on $\mathcal W_k$. 
Using the results of this section, we may proceed analogously to  the classical (commutative) setup, to extract moduli spaces out of  extension groups of line bundles, by considering extension classes up to bundle isomorphism. 

\section{Moduli of bundles on noncommutative deformations} 
\label{sec:ncmoduli}

We define the notion of isomorphism of vector bundles 
on a noncommutative deformation of $W_k$. A similar definition was used for bundles on surfaces in \cite{BG}.

\begin{definition}
	\label{nciso}
	Let $E$ and $E'$ be vector bundles over $(\mathcal W_k, \sigma)$ defined by transition matrices $T$ and $T'$ respectively. An {\it isomorphism} between $E$ and $E'$ is given by a pair of matrices $A_U$ and $A_V$ with entries in $\mathcal A^\sigma (U)$ and $\mathcal A^\sigma (V)$, respectively, which are invertible with respect to $\star$ and such that
	\[
	T'= A_V \star T \star A_U.
	\]
\end{definition}

\begin{notation} 
Denoting by $ \Ext^1_{\mathcal Alg} (\mathcal A (j), \mathcal A (-j))$ the subset of formally algebraic extension classes, 
we denote by  $\mathfrak M_j (\mathcal W_k)$ the quotient 
$$\mathfrak M_j (\mathcal W_k)\ce  \Ext^1_{\mathcal Alg} (\mathcal A (j), \mathcal A (-j))/{\sim}$$
 consisting of those classes of formally algebraic vector bundles (Def.\thinspace\ref{formally-algebraic}), whose classical limit is a stable vector bundle of charge $j$. 
 Here $\sim$ denotes bundle isomorphism as in Def.~\ref{nciso} and following \cite{BGK2} stability means that the classical limit does not split on the $0$-th formal neighbourhood.
 
 \begin{remark} 
 In this text we will borrow the terminology from mathematical physics, where the quotients $\mathfrak{M}_j$ are called the classical moduli spaces and the quotients $\mathfrak{M}_j^\hbar$ are called the quantum moduli spaces. 
 We observe that in this setup the term ``moduli'' refers to taking a quotient modulo isomorphism. 
 However, these are neither fine nor coarse moduli spaces in the sense of algebraic geometry; nevertheless they do turn out to be quasiprojective varieties.
 The same terminology was previously used in \cite{BG, BGK2}.
 \end{remark}
 
We denote by $\mathfrak M_j^{\hbar^n} (\mathcal W_k,\sigma)$ the moduli of bundles obtained by imposing the cut-off $\hbar^{n+1} = 0$, that is, the superscript $\hbar^n$ means quantised to level $n$. 
\end{notation}

Note that $\mathfrak M_j (\mathcal W_k,\sigma) \ce \mathfrak M_j^{\hbar^0} (\mathcal W_k,\sigma)
 = \mathfrak M_j (W_k)$ recovers the classical moduli space obtained when $\hbar = 0$,
while
$\mathfrak M_j^{\hbar} (\mathcal W_k ,\sigma)$  denotes the moduli on the first order quantization,
which will be the focus of this work. Accordingly:

\begin{definition} \label{ncmoduli} We call 
$\mathfrak M_j (\mathcal W_k,\sigma) $  the {\it classical moduli space} and 
$\mathfrak M_j^{\hbar} (\mathcal W_k ,\sigma)$  the {\it quantum moduli space} of bundles on $\mathcal W_k$.
\end{definition}

\begin{lemma}\label{classicalmoduli} \cite[Thm.\thinspace 2.7]{BGS} 
The classical moduli spaces of vector bundles of rank 2 and splitting type $j$  on $W_k$ 
has dimension $4j-5$.
\end{lemma}

\begin{definition}\label{split}
The {\it splitting type} of a vector bundle $E$ on $(\mathcal W_k, \sigma)$ is the one of its classical limit \cite[Def.~5.2]{BG}. 
Hence, when the classical limit is an $\mathrm{SL}(2,\mathbb C)$ bundle, the splitting type of $E$ is the smallest integer $j$ such that $E$ can be written as an extension of $\mathcal A (j)$ by $\mathcal A (-j)$. 
\end{definition}

We fix a splitting type $j$ 
and look at rank $2$ bundles on the first formal neighbourhood $\ell^{(1)}$ of $\ell\simeq\mathbb P^1 \subset W_1$
together with their extensions  up to first order in $\hbar$.
We now calculate isomorphism classes. 
Let $p + p' \hbar$ and $q + q' \hbar$ be two extension classes in $\Ext^1_{\mathcal A} (\mathcal A (j), \mathcal A (-j))$
 which are of splitting type $j$, {\it i.e.}\ in canonical $U$-coordinates $p, p', q, q'$ are multiples of $u_1,u_2$. 

According to Def.~\ref{nciso}  bundles defined by $p + p' \hbar$ and $q + q' \hbar$ are isomorphic, if there exist invertible matrices
\begin{equation*}
\begin{pmatrix}
a + a' \hbar & b + b' \hbar \\
c + c' \hbar & d + d' \hbar
\end{pmatrix}
\quad\text{and}\quad
\begin{pmatrix}
\alpha + \alpha' \hbar & \beta  + \beta'  \hbar \\
\gamma + \gamma' \hbar & \delta + \delta' \hbar
\end{pmatrix}
\end{equation*}
whose entries are holomorphic on $U$ and $V$, respectively, such that
%{\small \begin{align}
\begin{equation}
\label{equivalencefirstordermoduliW1}
\begin{pmatrix}
\alpha + \alpha' \hbar & \beta  + \beta'  \hbar \\
\gamma + \gamma' \hbar & \delta + \delta' \hbar
\end{pmatrix}
\star
\begin{pmatrix}
z^j & q + q' \hbar \\
 0  & z^{-j}
\end{pmatrix}
=
\begin{pmatrix}
z^j & p + p' \hbar \\
 0  & z^{-j}
\end{pmatrix}
\star
\begin{pmatrix}
a + a' \hbar & b + b' \hbar \\
c + c' \hbar & d + d' \hbar
\end{pmatrix}.
\end{equation}
%\end{align}}
We wish to determine the constraints such an isomorphism imposes on the coefficients of $q$ and $q'$. This is more conveniently rewritten by multiplying  by the right-inverse of $\left( \begin{smallmatrix} z^j & q + q' \hbar \\  0  & z^{-j} \end{smallmatrix} \right)$,  which (modulo $\hbar^2$) is
\[
\begin{pmatrix}
z^{-j} & -q - q' \hbar + 2 z^{-j} \{ z^j, q \} \hbar\\
   0   & z^j
\end{pmatrix}.
\]

We have that the zero section $\ell\simeq\mathbb P^1$ is cut out inside $W_k$ by $u_1=u_2=0$. 
Hence,   the $n$-th formal neighbourhood of $\ell$ is by definition
$\displaystyle \ell^{(n)}= \frac{\mathcal O_{W_1}}{\mathcal I^{n+1}}$ where $\mathcal I=<u_1,u_2>$.
So,
on $\ell^{(1)}$ we have that 
$u_1^2=u_2^2=u_1u_2=0$ and therefore we may write
$$
a = a_0 + a_1^1 u_1+a_1^2u_2,\quad
\alpha = \alpha_0 + \alpha_1^1 u_1+\alpha_1^2 u_2,$$
etc., 
where $a_1^i$, $\alpha_1^i$, etc.\ are holomorphic functions of $z$.

 Given that the $\star$ product is just usual multiplication of functions in the absence of $\hbar$, 
 if we make $\hbar = 0$ in Equation \eqref{equivalencefirstordermoduliW1}, we are left with the same equations that determine the isomorphism in the classical case, as in the proof of \cite[Prop.~3.2]{G}.
 Intuitively, we may think of the calculations in powers of $\hbar$ analogously to the calculations in formal neighbourhoods of $\ell$.
 
 Following the details of the proof of \cite[Prop.\thinspace3.2]{G}, by comparing the terms on the matrix equation \eqref{equivalencefirstordermoduliW1}, we realize that $a_0 = \alpha_0$, $d_0 = \delta_0$ are constant and $b = \beta = 0$.  
 Since we already know that on the classical limit the only equivalence on $\ell^{(1)}$ is  projectivization \cite[Prop.~3.2]{G}, 
 we assume $p = q$, keeping in mind  a projectivization  to be done in the end. We may also assume that the determinants of the changes of coordinates on the classical limit are $1$. Accordingly, we rewrite (\ref{equivalencefirstordermoduliW1}) as:
\begin{equation}
{\small \begin{pmatrix}
\alpha + \alpha' \hbar & \beta' \hbar \\
\gamma + \gamma' \hbar & \delta + \delta' \hbar
\end{pmatrix} 
\!=\!
\begin{pmatrix}
z^j & p + p' \hbar \\
 0  & z^{-j}
\end{pmatrix}
\star
\begin{pmatrix}
a + a' \hbar & b' \hbar \\
c + c' \hbar & d + d' \hbar
\end{pmatrix}
\star
 \begin{pmatrix}
z^{-j} & -p 
-q'\hbar  +2  \{ z^j, p \}z^{-j} \hbar \\
 0  & z^j
 \label{simplifiedW1}
\end{pmatrix}}
\end{equation}
where $a_0 = d_0 = \alpha_0 = \delta_0 = 1$.

Since we already know the moduli in the classical limit,
we only need to study terms containing $\hbar$, which after multiplying are:
\begin{align*}
(1,1) &= a' +\{ z^j a, z^{-j} \} + \{ z^j, a \} z^{-j} +\{ p c, z^{-j} \} + \{ p, c \} z^{-j} + (p c' + p' c) z^{-j} \\
(1,2) &= \{ p, d \}z^j -\{ a, p \}z^j - \{ z^j, a \} p + \{ z^j, p \} a + \{ p d, z^j \} +  2 z^{-j} \{ z^j, p \} p c
 + z^{2j} b' \\&\quad  - (p a' + q' a)z^j+ (p d' + p' d)z^j - (p c' + p' c + q' c) p \\
 (2,1) &= z^{-2j} c' \\
(2,2) &= d' + \{ z^{-j} d, z^j \} + \{ z^{-j}, d \} z^j - \{ z^{-j} c, p \} - \{ z^{-j}, c \} p- (p c'+ q' c) z^{-j}   + 2 \{ z^j, p \} z^{-2j} c .&
\end{align*}

 All four terms must be adjusted using the free variables to only contain expressions which are holomorphic on $V$ to satisfy (\ref{simplifiedW1}). 
 For example, in the $(2,1)$ term this condition is satisfied precisely when $c'$ is a section of $\mathcal O_{W_k} (2j)$.
 Computing  Poisson brackets, we see that the $(1,1)$ and $(2,2)$ terms can always be made holomorphic on $V$  by appropriate choices of  
 $c$ and $d'$, leaving the coefficients of $a'$ free. We will need to use these free coefficients for the next step.
 
It remains to analyse the $(1,2)$ term. 
Because we are working on the first formal neighbourhood of $\ell$, terms 
in $u_1^2, u_1u_2, u_2^2$  or higher  vanish (recall that we assume that $p, p', q'$ are multiples of $u_1$ or $u_2$).
 Since $z^{2j} b'$ is there to cancel out any possible terms having power of $z$ greater or equal to $2j$, we 
remove it from the expression, keeping in mind that we only need to cancel out the coefficients of the monomials $z^iu_1$ and $z^iu_2$ with 
$i \leq 2j-1$
in the expression:
\begin{equation}\tag{$\trevo$}
	\label{12term}
	(1,2) =  \{ p, d +a\} z^j- \{ z^j, a \} p + \{ z^j, p \} a + \{ p d, z^j \} +2 z^{-j} \{ z^j, p \} p c
	+ p(d'-a')z^j + (  p' -q' )z^j.
\end{equation}

To determine the quantum moduli spaces, we must verify what  restrictions are imposed on $q'$  so that 
$p'$ and $q'$  define isomorphic  bundles. Since this requires computing  brackets, 
the analysis must be carried out separately for each noncommutative deformation.
  
\section{Quantum moduli of bundles on $\mathcal W_1$}
 The Calabi--Yau threefold we consider in this section is the crepant resolution of the conifold singularity $xy-zw=0$, that is, 
 $$W_1\ce \Tot (\mathcal O_{\mathbb P^1}(-1) \oplus \mathcal O_{\mathbb P^1}(-1)) .$$

 We will carry out calculations using the canonical coordinates 
 $W_1= U \cup V$ where $U \simeq \mathbb C^3 \simeq V$ with  $U = \{z,u_1,u_2\}$, $V= \{\xi, v_1,v_2\}$,
 and change of coordinates on $U\cap V \simeq \mathbb C^* \times \mathbb C\times \mathbb C$ 
given by 
\[ \bigl\{ \xi = z^{-1} \text{ ,\quad}
         v_1 = z u_1 \text{ ,\quad}
         v_2 =  zu_2 \bigr\} \text{ .} \]
%so that $z = \xi^{-1}$, $u_1 = \xi^2 v_1$, and $u_2 =  v_2$.
Consequently, global functions on $W_1$ are generated over   $\mathbb C$ by  the monomials $1, u_1,zu_1, u_2,zu_2$.

  For each specific noncommutative deformation $(\mathcal W_1, \mathcal A^\sigma)$,
 we wish to compare the quantum and classical moduli spaces of vector bundles, see Def.\thinspace\ref{ncmoduli}. 
 This is part of the general quest to understand how deformations of a variety affect moduli of bundles on it,
 and it is worth noting  that  no commutative deformation of $W_1$ is known to exits.

  For a rank $2$ bundle $E$ on a noncommutative deformation  $\mathcal W_1$  with a  canonical  matrix
$\tiny	\begin{pmatrix}
		z^j & {\bf p} \\
		0  & z^{-j}
	\end{pmatrix}$
as in (\ref{ncpoly}) where  ${\bf p} = \sum_{n=0}^\infty p_n \hbar^n$, expression (\ref{poli1})  gives us the general form of the coefficients $p_n$. In particular,  on the first formal neighbourhood, we have: 
\begin{equation}p=\sum_{l=-j+2}^{j-1}p_{l10}z^lu_1+ \sum_{l=-j+2}^{j-1}p_{l01}z^lu_2,\end{equation}
  where $p=0$ if $j=1$.

Each noncommutative deformation comes from some Poisson structure which determines the first order 
terms of the corresponding star product, see Sec.~\ref{ncsec}. The most basic 
Poisson structures  $\sigma$ on $W_1$ are those which generate all others 
over global functions.
We call these generators the {\it basic} Poisson structures.

\begin{theorem}\label{iso1} If  $\sigma $ is a basic Poisson structure on $W_1$,
then the quantum moduli space $\mathfrak M^{\hbar}_j(\mathcal W_k, \sigma)$
and its classical limit  are isomorphic:
 $$\mathfrak M^{\hbar}_j(\mathcal W_1, \sigma)\simeq \mathfrak M_j(W_1)\simeq \mathbb P^{4j-5}.$$
 \end{theorem}

\begin{proof}
We perform the computations using the bracket $ \sigma_1  = \partial z\wedge \partial u_1$; the 
choice of such a generator is irrelevant, since all the 4 generators give pairwise isomorphic Poisson manifolds. 
To obtain an isomorphism, we need to cancel out all coefficients of the terms 
$$z^{2} u_1, \dotsc, z^{2j-1} u_1 \quad\text{and}\quad z^2u_2, \dotsc, z^{2j-1} u_2$$
appearing in expression \ref{12term}.
Calculating   $\sigma_1 $   brackets, we have
$\displaystyle{ \{ z^j, f \} = jz^{j-1}\frac{\partial f}{\partial u_1}}$, 
and following expressions for $a$ and $d$ coming from 
the classical part
$$a=1+a_1^1u_1+a_1^2u_2, \quad d= 1-a_1^1u_1-a_1^2u_2,$$
where $a_1^i$ and $d_1^i$ are functions of $z$, gives
$\displaystyle\frac{\partial a}{\partial u_1} =a_1^1, \quad \frac{\partial d}{\partial u_1}= -a_1^1,$
so that $ \{ p, d \}-\{ a, p \} z^j = -2	 \left[ \frac{\partial p}{\partial z} a_1^1 - \frac{\partial a}{\partial z} \frac{\partial p}{\partial u_1}\right]z^j$.
Therefore, expression \ref{12term} becomes
	$$	\trevo = 
%	\left[ -\frac{\partial p}{\partial z} a_1^1- \frac{\partial d}{\partial z} \frac{\partial p}{\partial u_1}\right] z^j-
-2	 \left[ \frac{\partial p}{\partial z} a_1^1 - \frac{\partial a}{\partial z} \frac{\partial p}{\partial u_1}\right]z^j
 + 2j \left[   \frac{\partial p}{\partial u_1} (a_1^1u_1+a_1^2u_2+pc)\right] z^{j-1}+  p(d'-a') z^j + (  p' -q' )z^j .$$

Now we need to cancel out separately the coefficients of each monomial $z^i u_1$ and $z^i u_2$ for $2 \leq i\leq 2j-1$, that is, 
all those terms potentially giving nonholomorphic functions.
To determine the classes in the moduli space we need to verify what constraints are imposed on $q'$.
Take for instance the monomial $z^iu_1$ in $ (  p' d-q'a )z^j$.
Since $a'$ remains free  we can always  choose its corresponding coefficient in order to 
cancel out the  term in $z^iu_1$ in the entire expression of $(1,2)$.
Indeed, notice that the expressions $ p(d'-a')z^j$  and  $(  p' d-q'a )z^j$ contain monomials of the same orders, 
all of which may be adjusted to zero by choosing $a'$. Moreover 
the first three summands in $\trevo$ also contain the same list of monomials, hence may also be absorbed 
by the appropriate choices of coefficients of $a,a'$ and $c$.

Since this process can be independently carried out for each monomial, we then conclude that 
the expression $\trevo$ can be made holomorphic on $V$ for any choice of $q'$. Hence, there are no 
restrictions on $q'$. Thus, we obtain an equivalence $p+p'\hbar \sim p+q'\hbar$ for all $q'$ and the projection onto the
classical limit (the first coordinate) 
$$\pi_1 \colon  \mathfrak M^{\hbar}_j(\mathcal W_1, \sigma)\rightarrow \mathfrak M_j(W_1)
$$
taking $(p,p')$ to $p$ is an isomorphism. The isomorphism type of the moduli space is given in \cite[Lem.\thinspace 6.2]{BGS}
as $\mathbb P^{4j-5}$.
\end{proof}

We now calculate the quantum moduli space for the particular choice of splitting type  $j=2$ and for a 
different choice of Poisson structure on $W_1$.
 We  use the notation $p \in \mathfrak M_j (W_1)$ to refer to a point in the classical moduli space, that is, a 
 rank 2 bundle is labelled by its extension class.

\begin{example}[$j = 2$ and $\sigma = u_1\sigma_1$]
\label{m2u}
Here we  write $$p =p_0 zu_1+p_1  u_1
+ p_2 zu_2+ p_3  u_2, 	\quad p' =p'_0 zu_1+p'_1  u_1+ p'_2 zu_2+ p'_3  u_2.	$$
for the first order part of the extension class, where we have renamed the coefficients 
to simplify notation ( $p_0\ce p_{110}, p_1\ce p_{010},p_2 \ce p_{101}, p_3\ce  p_{001}$).
Lem.~\ref{cone} implies that $\sigma = u_1\sigma_1$ is also a Poisson structure on $W_1$. 
With this choice, all brackets acquire an extra $u_1$ in comparison to the 
bracket $\sigma_1$  used in the proof of Thm.~\ref{iso1}, so that in the first formal neighbourhood  the $(1,2)$-term described in \ref{12term}  simplifies to just:
$$
\trevo= z^2p(d'-a')+z^2(p'-q').
$$
Here 
$a'=a'_0+{a'}_1u_1+{a'}_2u_2, \quad d'= d'_0-{d'}_1u_1-{d'}_2u_2,$
so that 
$$d'-a' ={(d'-a')}_0 +{(d'-a')}_1u_1+{(d'-a')}_2u_2.$$
Hence, the total expression of \ref{12term} is
\begin{eqnarray*}
\trevo&=&(p_{0} z^3 u_1+p_{1}  z^2u_1
+ p_{2} z^3 u_2+ p_{3}z^2  u_2)((d'-a')_0 +{(d'-a')}_1u_1+{(d'-a')}_2u_2))
\\
&&+(p'_0-q'_0) z^3 u_1+(p'_1-q'_1)  z^2u_1+ (p'_2-q'_2) z^3 u_2+ (p'_3-q'_3)z^2  u_2,
\end{eqnarray*}
where we canceled out all the monomials containing $u_1^2, u_1u_2,$ and $u_2^2$, since we work on the 
first formal neighbourhood.
We rename $(d'-a')_0(z)= \lambda_0+\lambda_1z+\lambda_2z^2+ \ldots$ to simplify notation, 
and since all terms in $(1,2)$ having powers of $z$ equal to $4$ and higher can be cancelled out by 
the appropriate choice of the $z^{2j}b'$, it suffices to analyse the expression
\begin{eqnarray*}
\trevo &=& (  p_{0} z^3 u_1+p_{1}  z^2u_1
+  p_{2} z^3 u_2+  p_{3}z^2  u_2)(\lambda_0+\lambda_1z)\\
& & +(q'_{0}-p'_{0}) z^3 u_1+ (q'_{1}-p'_{1})  z^2u_1+ (q'_{2}-p'_{2}) z^3 u_2+ (q'_{3}-p'_{3})z^2  u_2.
\end{eqnarray*}
%%%%%%%%%%%%%%
To have an isomorphism $q' \sim p'$, we need to cancel out  the coefficients of $ z^3u_1, z^2u_1,$ $z^3u_2,z^2u_2$
in $\trevo$  with appropriate choices of $\lambda_i$.
Consequently, $q' \sim p'$ if and only if the following equality holds for some choice of $\lambda_0$ and $\lambda_1$:
$$
\begin{pmatrix}
q'_{0}- p'_{0} \\
q'_{1}- p'_{1} \\
q'_{2}- p'_{2} \\
q'_{3}- p'_{3} \\
\end{pmatrix}
= 
\lambda_0 
\begin{pmatrix}
p_{0}\\
p_{1} \\
p_{2}\\
p_{3}\\
\end{pmatrix}+
\lambda_1
\begin{pmatrix}
p_{1} \\
0 \\
p_{3}\\
0
\end{pmatrix}.
$$

When the vectors $v_1= (p_0,p_1,p_2,p_3)$ and $v_2= (0,p_1,0,p_3)$ 
are linearly independent, the point $q'$ belongs to the plane that passes through the point $p'$
with $v_1$ and $v_2$ as direction vectors.

Therefore, whenever $v_1$ and $v_2$ are linearly 
independent vectors, the fibre over $p= (p_0,p_1,p_2,p_3)$ is a copy of $\mathbb C^4$ foliated by $2$-planes. 
The
leaf containing  a  point $p'$  forms the equivalence class of $p'$. 
Thus, the moduli space over the fibre
 over $p$ is parametrised by the 2-plane through the origin in the direction perpendicular to $v_1,v_2$
over the point  $p$, except when $p_1=p_3=0$.

In contrast, the fibre over 
a point $p=(p_0,0,p_2,0)$ is a copy of $\mathbb C^4$ foliated by lines in the direction of $v_1= (p_0,0,p_2,0)$.
In this case, the moduli space over $p$ is   parametrised by a copy of $\mathbb C^3$  perpendicular to $v_1$.

We conclude that $  \mathfrak M^{\hbar}_2(\mathcal W_1, \sigma)\rightarrow \mathfrak M_2(W_1)\simeq \mathbb P^3
$ (where the isomorphism is given by Lem.~\ref{classicalmoduli})
is the \'etale space of a constructible sheaf, whose stalks have  
\begin{itemize}
\item  dimension  2 over the Zariski open set 
$(p_1, p_3)\neq (0,0)$,
and
\item  dimension 3  over the $\mathbb P^1$  cut out by $p_1=p_3=0$ in $\mathbb P^3$.
\end{itemize}
\end{example}

The same techniques readily generalise to give a  description of the quantum moduli spaces for other choices of noncommutative deformations.

\begin{theorem}\label{e1}
If $\sigma$ is an extremal Poisson structure on $W_1$, then the quantum moduli space $\mathfrak M_j^{\hbar} (\mathcal W_1, \sigma)$ can 
be viewed as the étale space of a constructible sheaf of generic rank $2j-2$ over the classical moduli space 
$\mathfrak M_j (W_1)$ with singular stalks up to   rank $4j -5$.
\end{theorem}

\begin{proof} We give the details of the case $j=3$, for an extremal Poisson structure, that is, the 
case when all brackets vanish on the first formal neighbourhood.
 The general case is clear from these calculations, just notationally more complicated (one uses the same method of comparing monomials of the same order, but a larger $j$ implies a significantly larger number of monomials). 

When $j=3$ and $\sigma= u_1\sigma_1$, expression \ref{12term} becomes: 
$$\trevo =  p(d'-a')z^3 + (  p' d-q'a )z^3,$$
and we get a system of equations:
\begin{eqnarray*}
\trevo&=&\left(p_0 z^5u_1 + p_1z^4u_1 + p_2 z^3u_1+ p_{3}z^2u_1+p_4z^{5}u_2+p_5z^4u_2+p_6z^3u_2 + p_7z^2u_2
\right)\\
&&\cdot( \lambda_0 +\lambda_1z +\lambda_2z^2+ \lambda_3z^3 + \lambda_{4}z^{4}) +\\
&&+(p'-q')_0 z^5u_1 + (p'-q')_1z^4u_1 + (p'-q')_2 z^3u_1+ (p'-q')_{3}z^2u_1\\
&&+(p'-q')_4z^{5}u_2+(p'-q')_5z^4u_2+(p'-q')_6z^3u_2 + (p'-q')_7z^2u_2.
\end{eqnarray*}

To have an isomorphism $q' \sim p'$, we need to cancel out  the coefficients of $ z^5u_1, z^4u_1, z^3u_1, z^2u_1,$ $ z^5u_2, z^4u_2,z^3u_2,z^2u_2$
in $\trevo$  with appropriate choices of $\lambda_i$.
Consequently, $q' \sim p'$ if and only if the following equality holds for some choice of 
$\lambda_0, \lambda_1, \lambda_2,\lambda_3$:
 $$
\begin{pmatrix}
q'_{0}- p'_{0} \\
q'_{1}- p'_{1}\\
q'_{2}- p'_{2}\\
q'_{3}- p'_{3}\\
q'_{4}- p'_{4}\\
q'_{5}- p'_{5}\\
q'_{6}- p'_{6}\\
q'_{7}- p'_{7}\\
\end{pmatrix}
=
\lambda_0
\begin{pmatrix}
p_{0} \\
p_{1}\\
p_{2}\\
p_3\\
p_4\\
p_{5}\\
p_{6}\\
p_{7}\\
\end{pmatrix}+
\lambda_1
\begin{pmatrix}
p_{1}\\
p_{2}\\
p_3\\
0\\
p_5\\
p_{6}\\
p_{7}\\
0
\end{pmatrix}+
\lambda_2
\begin{pmatrix}
p_{2}\\
p_3\\
0\\
0\\
p_6\\
p_{7}\\
0 \\
0\\
\end{pmatrix}
+
\lambda_3
\begin{pmatrix}
p_{3}\\
0\\
0\\
0\\
p_{7}\\
0 \\
0\\
0\\
\end{pmatrix}.
$$
Consider now the family $\mathfrak U$ of vector spaces over  $\mathfrak M_2(W_1)\simeq \mathbb P^7$ whose
fibre at $p$ is given by 
$$\mathfrak U_p=
\left(
\begin{matrix}
p_{0} & p_1 & p_2& p_3 \\
p_{1}& p_2 & p_3& 0 \\
p_{2}& p_3 &0&0\\
p_3 &0 &0 & 0\\
p_4 &p_5& p_6&  p_7\\
p_{5}& p_6 & p_7& 0 \\
p_{6}& p_7 & 0 & 0\\
p_{7}& 0& 0& 0\\
\end{matrix}
\right).
$$
Now, the quantum moduli space is obtained from this family after dividing by 
the equivalence relation $\sim$ over each point $p$.
Hence $$ \mathfrak M^{\hbar}_2(\mathcal W_1, \sigma) = \mathfrak U/ \sim.$$

We conclude that $  \mathfrak M^{\hbar}_2(\mathcal W_1, \sigma)\rightarrow \mathfrak M_2(W_1)\simeq \mathbb P^7
$ (where the isomorphism is given by Lem.~\ref{classicalmoduli})
is the \'etale space of a constructible sheaf or rank 4, with stalk at $p$  having  dimension equal to the 
corank of $\mathfrak U_p$, in this case
 $$4 \leq \dim {\mathfrak M^{\hbar}_2(\mathcal W_1, \sigma)}_ p = 8-\rk \mathfrak U_p \leq 7.$$
 In the general case we have 
  $$2j-2 \leq \dim {\mathfrak M^{\hbar}_j(\mathcal W_1, \sigma)}_ p =\mbox{corank}\, \mathfrak U_p= 
  %4j-4 -\rk \mathfrak U_p
   \leq 4j-5 .$$
   \end{proof}

\section{Quantum moduli  of bundles on  $\mathcal W_2$}

 The Calabi--Yau threefold we consider in this section is a crepant resolution of the singularity $xy-w^2=0$ in $\mathbb C^4$, that is
 $$W_2\ce \Tot (\mathcal O_{\mathbb P^1}(-2) \oplus \mathcal O_{\mathbb P^1}) = Z_2 \times \mathbb C.$$
 
 Similarly to what we did for $W_1$, we will carry out calculations using the canonical coordinates 
 $W_2= U \cup V$ where $U \simeq \mathbb C^3 \simeq V$ with  $U = \{z,u_1,u_2\}$, $V= \{\xi, v_1,v_2\}$,
 and change of coordinates on $U\cap V \simeq \mathbb C^* \times \mathbb C\times \mathbb C$ 
given by 
\[ \bigl\{ \xi = z^{-1} \text{ ,\quad}
         v_1 = z^2 u_1 \text{ ,\quad}
         v_2 =  u_2 \bigr\} \text{ .} \]
%so that $z = \xi^{-1}$, $u_1 = \xi^2 v_1$, and $u_2 =  v_2$.
Consequently, global holomorphic functions on $W_2$ are generated by $1, u_1,zu_1, z^2u_1,u_2$.

  For each
specific noncommutative deformation $(\mathcal W_2, \mathcal A^\sigma)$,
 we wish to compare the quantum and classical moduli spaces of vector bundles, see Def.\thinspace\ref{ncmoduli}. 

For a rank $2$ bundle $E$ on a noncommutative deformation $\mathcal W_2$  with a  canonical  matrix
$\tiny	\begin{pmatrix}
		z^j & {\bf p} \\
		0  & z^{-j}
	\end{pmatrix}$
as in (\ref{ncpoly}) where  ${\bf p} = \sum_{n=0}^\infty p_n \hbar^n$, expression (\ref{poli1})  gives us the general form of the coefficients $p_n$. In particular,  on the first formal neighbourhood, we have: 
\begin{equation}p=\sum_{l=-j+3}^{j-1}p_{l10}z^lu_1+\sum_{l=-j+1}^{j-1}p_{l01}z^lu_2 \end{equation}
where in case $j=1$ we have only $p_{001}u_2$.

To describe the quantum moduli for Poisson structures on $\mathcal W_2$, 
we consider the expression \ref{12term}:
$$	\trevo = \{ p, d +a\} z^j  - \{ z^j, a \} p + \{ z^j, p \} a  +  \{ p d, z^j \} +2 z^{-j} \{ z^j, p \} p c
+ p(d'-a')z^j + (  p' d-q'a )z^j,
$$
where we need to cancel out the coefficients of
$z^{3} u_1, \dotsc, z^{2j-1} u_1 \quad\text{and}\quad zu_2, \dotsc, z^{2j-1} u_2.$

Each noncommutative deformation comes from some Poisson structure. The most basic 
Poisson structures  $\sigma$ on $W_2$ are those which generate all others 
over global functions.
We call these generators the {\it basic} Poisson structures.
Now, we compute the quantum moduli of bundles for them.

\begin{remark*} We observe that the 4 Poisson manifolds $(\mathcal W_2, \sigma_i)$ for $i=1,2,3,4,$ are pairwise nonisomorphic. 
This can be verified by observing that they have have inequivalent degeneracy loci (formed by those symplectic leaves that consist of single points). 
The following table depicts the toric diagrams of their degeneracy loci:
$$
\begin{tabular}{c|c}
\multicolumn{2}{c}{\sc $W_2$ Poisson structures }\\
\multicolumn{2}{c}{} \\
bracket & degeneracy  \\ \hline
$\sigma_1$ & \ccfan \\ \hline
$\sigma_2$  & $\emptyset$\\ \hline
$\sigma_3$  & \Zzero \\ \hline
$\sigma_4$ &  \ccfan $\cup$ \cfan \\ \hline
\end{tabular}
$$
Nevertheless, the 4 quantum moduli spaces defined by these basic Poisson structures turn out to be all isomorphic. 
\end{remark*}

\begin{theorem}\label{iso2} If  $\sigma $ is a basic Poisson structure on $W_2$,
then the quantum moduli space $\mathfrak M^{\hbar}_j(\mathcal W_k, \sigma)$
%of rank 2 vector bundles on $(\mathcal W_2, \mathcal A^\sigma)$ with splitting type $j$ 
and its classical limit  are isomorphic:
 $$\mathfrak M^{\hbar}_j(\mathcal W_2, \sigma)\simeq \mathfrak M_j(W_2)\simeq \mathbb P^{4j-5}.$$
 \end{theorem}
 
\begin{proof}\label{exW2e2}
We carry out 
calculations for the basic bracket $\sigma_4= u_1\partial u_1\wedge \partial u_2.$ 
It does turn out that the result is the same for the the basic brackets.
The calculation for $\sigma_4$ is shorter, since  any of the brackets having one entry equal to $z^j$  vanishes. 
Because we work on the first formal neighbourhood, we also remove the expressions that 
are quadratic in the $u_i$ variables. 

So, the expression \ref{12term} that remains to be analysed simplifies to:
$$\trevo = \{ p, d +a\}z^j  	+ p(d'-a')z^j + (  p' d-q'a )z^j,$$
where we must cancel out the coefficients of the monomials
$
z^{3} u_1, \dotsc, z^{2j-1} u_1 \, \text{and}\, zu_2, \dotsc, z^{2j-1} u_2.
$
On the first formal neighbourhood, we write 
$$a= 1 + a_1(z) u_1+a_2(z)u_2,\quad d= 1 + d_1(z) u_1+d_2(z)u_2, \quad \text{and}$$
$$a'= a'_0(z) + a'_1(z) u_1+a'_2(z)u_2,\quad d'= d'_0(z) + d'_1(z) u_1+d'_2(z)u_2,$$
so that the partials are
$$\partial u_i a=  a_i(z) \quad \partial u_i d=  d_i(z) \quad \text{and} \quad \partial u_2 a=  a_2(z) \quad \partial u_2 d=  d_2(z) . $$
The extension class given in  (\ref{poli2})   becomes
$\displaystyle p =\sum_{l=3-j}^{j-1}p_{l10}z^lu_1 + \sum_{l=1-j}^{j-1}p_{l01}z^lu_2,  $
and computing the bracket gives
$$ \{ p, d + a \}z^j=  \left(\sum_{l=3-j}^{j-1}p_{l10}z^l \right)(d_2(z)+a_2(z))z^ju_1 + \left(\sum_{l=1-j}^{j-1}p_{l01}z^l\right)(d_1(z)+a_1(z))z^ju_1.$$

To work with a simpler notation, we present details  of  \ref{12term} 
 when  $j=2$, in which  case we can express the extension class as 
 $$p=p_0zu_1 +p_1zu_2+ p_2u_2+p_3z^{-1}u_2,$$
having renamed the coefficients  for simplicity
(making  $p_0\ce p_{110} ,
p_1\ce p_{101}, 
p_2\ce p_{001},
p_3\ce p_{-101}$). We will point out the steps for   generalising to higher $j$.

Assuming $j=2$, we have
$$\{ p, d + a \}z^2=  p_0(d_2(z)+a_2(z))z^3u_1 + (p_{1}z^3+p_2z^2 + p_3z)(d_1(z)+a_1(z))u_1 .$$
To obtain equivalence between $q'$ and $p'$, we  must cancel out coefficients of  $z^3u_1, zu_2, z^2u_2, z^3u_2$ in the expression of \ref{12term},
which becomes
 \begin{eqnarray*}
\trevo&=&p_0(d_2(z)+a_2(z))z^3u_1 + (p_{1}z^3+p_2z^2 + p_3z)(d_1(z)+a_1(z))u_1\\
&& +(p_0z^3u_1 +p_1z^3u_2+p_2z^2u_2+ p_3zu_2)
(d'_0(z)-a'_0(z)) \\
&&+(p'_0z^3u_1 +p'_1z^3u_2+p'_2z^2u_2+ p'_3zu_2)
 \\
&&-(q'_0z^3u_1+q'_1z^3u_2 +q'_2z^2u_2+ q'_3zu_2) .
 \end{eqnarray*}
Since the highest power of $z$ to be considered is $3$,  we observe that  $d_2(z) +a_2(z)$ may be chosen  conveniently, 
we cancel out all terms in $z^3 u_1$. We may also choose $d_1(z)+a_1(z)=0$,  leaving 
 \begin{eqnarray*}
\trevo &=& (p_1z^3u_2+p_2z^2u_2+ p_3zu_2)
(d'_0(z)-a'_0(z)) \\
&&+(p'_1z^3u_2+p'_2z^2u_2+ p'_3zu_2) \\
& &-(q'_1z^3u_2 +q'_2z^2u_2+ q'_3zu_2).
 \end{eqnarray*}
 Now we may  choose $d'_0-a'_0$ appropriately to cancel out all terms in $u_2$. 
 We conclude that there are no conditions imposed on $q'$. In other words, here $p+p'\hbar$ is equivalent to  $p+q'\hbar $ 
 for any choice of  $q'$. Hence, the quantum and classical  moduli spaces 
 are isomorphic. 
 
 The generalisation to higher $j$ works out similarly, we can first choose $d_i+a_i$ for $i>0$ to cancel out the 
 coefficients of $u_1$ and then choose $d'_0-a'_0$ to take care of the coefficients of $u_2$. So, for all $j$ 
 using the bracket $\sigma_4$ we conclude that 
 the quantum and classical  moduli spaces are isomorphic
 $$\mathfrak M^{\hbar}_j(\mathcal W_2, \sigma_4)\simeq \mathfrak M_j(W_2)\simeq \mathbb P^{4j-5}$$
 where the second isomorphism is proven in \cite[Prop.\thinspace 3.24]{K}.
  \end{proof}

\begin{example} Now choose any Poisson structure of $W_2$ for which all brackets in \ref{12term} 
vanish on neighbourhood 1, for example $\sigma=u_1\sigma_4 = u_1^2\partial u_1\wedge \partial u_2$ works. 
In such a case, the expression for \ref{12term} reduces to:
$$\trevo =  p(d'-a')z^j + (  p' d-q'a )z^j.$$
Now, consider  the case of $j=2$, when we have:
 \begin{eqnarray*}
\trevo &=& (p_0z^3u_1 +p_1z^3u_2+p_2z^2u_2+ p_3zu_2)+
(d'_0(z)-a'_0(z)) \\
&&+(p'_0z^3u_1 +p'_1z^3u_2+p'_2z^2u_2+ p'_3zu_2)
 \\
&&- (q'_0z^3u_1+q'_1z^3u_2 +q'_2z^2u_2+ q'_3zu_2).
 \end{eqnarray*}
Setting
$$d'_0(z)-a'_0(z)= \lambda_0 +\lambda_1z +\lambda_2z^2,$$
we get a system of equations:
$$
\begin{pmatrix}
q'_{0}- p'_{0} \\
q'_{1}- p'_{1}\\
q'_{2}- p'_{2} \\
q'_{3}- p'_{3}\\
\end{pmatrix}
=
\begin{pmatrix}
\lambda_0 & 0& 0& 0 \\
 0 & \lambda_0 &\lambda_1& \lambda_2\\
 0&0& \lambda_0 &\lambda_1\\
 0&0&0&\lambda_0
\end{pmatrix}
\begin{pmatrix}
p_{0} \\
p_{1}\\
p_{2}\\
p_{3}\\
\end{pmatrix}.
$$

Since we can choose $\lambda_1$ and $\lambda_2$ to solve the second and third equations, we see that 
$q'_1$ and $q'_2$ are free. Hence
$(q'_0,q'_1,q'_2,q'_3) \sim \lambda_0(q'_0, *,*,q'_3),$
and our system of equations reduces to
$$
\begin{pmatrix}
q'_{0}- p'_{0} \\
q'_{3}- p'_{3}\\
\end{pmatrix}
=\lambda_0
\begin{pmatrix}
p_{0} \\
p_{3}\\
\end{pmatrix},
$$
which is  the parametric equation of a line in the $(q'_0,q'_3)$-plane
whenever $(p_0, p_3)\neq (0,0)$.
The entire question of moduli now reduces to the 2-dimensional case, disregarding $p_1,p_2$ coordinates.

If $(p_0, p_3)\neq (0,0)$, then  the equivalence class of $q'$ in the fibre over the point $p$ 
%of $\mathfrak M^{\hbar}_2(\mathcal W_2, \sigma)\rightarrow \mathfrak M_2(W_2)$ 
is the 1-dimensional subspace $L$ directed by the vector $(p_0, p_3)$ and passing through $(q'_0,q'_3)$
in the $(p_0',p_3')$-plane. 

If $p_0=p_3=0$, then we must have the equality $(q'_0,q'_3) = (p'_0,p'_3)$. 
So, its the equivalence class consists of a single point.

Accordingly, the set of equivalence classes over $p$ can be represented either 
by the line $L^\perp$ by the origin perpendicular to $L$ (directed by $(-p_3,p_0)$ when $(p_0, p_3)\neq (0,0)$
or else by the entire $(p'_0,p'_3)$-plane over $(0,0)$.

We conclude that $  \mathfrak M^{\hbar}_2(\mathcal W_2, \sigma)\rightarrow \mathfrak M_2(W_2)\simeq \mathbb P^3
$ (where the isomorphism is given by Lem.~\ref{classicalmoduli})
is the \'etale space of a constructible sheaf, whose stalks have  
\begin{itemize}
\item  dimension  1 over the Zariski open set 
$(p_0, p_3)\neq (0,0)$,
and
\item  dimension 2  over the $\mathbb P^1$  cut out by $p_0=p_3=0$ in $\mathbb P^3$.
\end{itemize}
In fact, we could express this moduli space as a sheaf given by an extension of $ \mathcal O_{\mathbb P^3}(+1)$
by  a torsion sheaf.
\end{example}

\begin{theorem}\label{e2}
If $\sigma$ is an extremal Poisson structure on $W_2$, then
 the quantum moduli space $\mathfrak M_j^{\hbar} (\mathcal W_2, \sigma)$ can 
be viewed as the étale space of a constructible sheaf  of generic rank $2j-3$ over the classical moduli space 
$\mathfrak M_j (W_2)$ with singular stalks up to   rank $4j -6$.
\end{theorem}

\begin{proof}
Now, for  $j=3$, we write down the extremal example when the brackets vanish 
on the first formal neighbourhood.

Where, assuming all brackets vanish on the first formal neighbourhood, 
we need to cancel out 
 the coefficients of
$z^{3} u_1, \dotsc, z^{2j-1} u_1 \quad\text{and}\quad zu_2, \dotsc, z^{2j-1} u_2$ in 
$$\trevo =  p(d'-a')z^j + (  p' d-q'a )z^j.$$
For $j=3$ we have
$$ p =\sum_{l=0}^{2}p_{l10}z^lu_1 + \sum_{l=-2}^{2}p_{l01}z^lu_2,  $$
which we rewrite as
$$ p=p_0 z^2u_1 + p_1zu_1 + p_2 u_1+ p_{3}z^{2}u_2+p_4z^{1}u_2+p_5u_2+p_6z^{–1}u_2 + p_7z^{-2}u_2.$$
Setting
$$d'_0(z)-a'_0(z)= \lambda_0 +\lambda_1z +\lambda_2z^2+ \lambda_3z^3 + \lambda_{4}z^{4},$$
expression
$$\trevo =  p(d'-a')z^3 + (  p' d-q'a )z^3$$
becomes
\begin{eqnarray*}
\trevo&=&\left(p_0 z^5u_1 + p_1z^4u_1 + p_2 z^3u_1+ p_{3}z^5u_2+p_4z^{4}u_2+p_5z^3u_2+p_6z^2u_2 + p_7zu_2
\right)\\
&&\cdot( \lambda_0 +\lambda_1z +\lambda_2z^2+ \lambda_3z^3 + \lambda_{4}z^{4}) \\
&&+
(p'-q')_0 z^5u_1 + (p'-q')_1z^4u_1 + (p'-q')_2 z^3u_1\\
&&+ (p'-q')_{3}z^5u_2+(p'-q')_4z^{4}u_2+(p'-q')_5z^3u_2+(p'-q')_6z^2u_2 + (p'-q')_7zu_2.	
\end{eqnarray*}

Then the solution of our problem is now 

\[
\begin{pmatrix}
q'_{0}- p'_{0} \\
q'_{1}- p'_{1}\\
q'_{2}- p'_{2}\\
q'_{3}- p'_{3}\\
q'_{4}- p'_{4}\\
q'_{5}- p'_{5}\\
q'_{6}- p'_{6}\\
q'_{7}- p'_{7}\\
\end{pmatrix}
=
\lambda_0
\begin{pmatrix}
p_{0} \\
p_{1}\\
p_{2}\\
p_{3}\\
p_{4}\\
p_{5}\\
p_{6}\\
p_{7}\\
\end{pmatrix}+
\lambda_1
\begin{pmatrix}
p_{1}\\
p_{2}\\
0\\
p_4\\
p_5\\
p_{6}\\
p_{7}\\
0
\end{pmatrix}+
\lambda_2
\begin{pmatrix}
p_{2}\\
0\\
0\\
p_5\\
p_6\\
p_{7}\\
0 \\
0\\
\end{pmatrix}+
\lambda_3
\begin{pmatrix}
0\\
0\\
0\\
p_6\\
p_7\\
0\\
0\\
0\\
\end{pmatrix}+
\lambda_4
\begin{pmatrix}
0\\
0\\
0\\
p_7\\
0\\
0\\
0\\
0\\
\end{pmatrix}=
\begin{pmatrix}
p_{0} & p_1 & p_2 & 0 & 0 \\
p_{1} & p_2 & 0   & 0 & 0\\
p_{2} & 0   & 0   & 0 & 0\\
p_3 & p_4 & p_5 & p_6 & p_7 \\
p_4 & p_5 & p_6 & p_7 & 0 \\
p_{5}& p_6 & p_7 & 0 & 0 \\
p_{6}& p_7 & 0 & 0 & 0\\
p_{7}& 0& 0 & 0 & 0\\
\end{pmatrix}
\begin{pmatrix}
\lambda_0 \\
\lambda_1 \\
\lambda_2 \\
\lambda_3 \\
\lambda_4
\end{pmatrix}.
\]

Notice that $\lambda_3$ and $\lambda_4$ can always be chosen to 
solve the equations involving $q'_3$ and $q'_4$ so that these 2 coordinates can take 
any value, that is, there are isomorphisms
\[
(q'_0,q'_1,q'_2,q'_3,q'_4,q'_5,q'_6,q'_7) \sim (q'_0,q'_1,q'_2, *,*,q'_5,q'_6,q'_7).
\]
Consequently, we may remove $q'_3$ and $q'_4$ and rewrite the reduced system as:
 \[
\begin{pmatrix}
q'_{0}- p'_{0} \\
q'_{1}- p'_{1}\\
q'_{2}- p'_{2}\\
q'_{5}- p'_{5}\\
q'_{6}- p'_{6}\\
q'_{7}- p'_{7}\\
\end{pmatrix}
=
\lambda_0
\begin{pmatrix}
p_{0} \\
p_{1}\\
p_{2}\\
p_{5}\\
p_{6}\\
p_{7}\\
\end{pmatrix}+
\lambda_1
\begin{pmatrix}
p_{1}\\
p_{2}\\
0\\
p_{6}\\
p_{7}\\
0
\end{pmatrix}+
\lambda_2
\begin{pmatrix}
p_{2}\\
0\\
0\\
p_{7}\\
0 \\
0\\
\end{pmatrix}=
\begin{pmatrix}
p_{0} & p_1 & p_2 \\
p_{1}& p_2 & 0 \\
p_{2}& 0&0\\
p_{5}& p_6 & p_7 \\
p_{6}& p_7 & 0 \\
p_{7}& 0& 0\\
\end{pmatrix}
\begin{pmatrix}
\lambda_0 \\
\lambda_1 \\
\lambda_2
\end{pmatrix}.
\]
 Here $q' \sim p'$ if and only if the  equality holds for some choice of $\lambda_0, \lambda_1, \lambda_2$.
 Consider now the family $\mathfrak U$ of vector spaces over $\mathfrak M_2(W_2)\simeq \mathbb P^7$ whose
fibre at $p$ is given by %(Hankel matrix)
\[
\mathfrak U_p=
\begin{pmatrix}
p_{0} & p_1 & p_2 \\
p_{1}& p_2 & 0 \\
p_{2}& 0&0\\
p_{5}& p_6 & p_7 \\
p_{6}& p_7 & 0 \\
p_{7}& 0& 0\\
\end{pmatrix}
=
\begin{pmatrix}
H_3(p_0, p_1, p_2, 0, 0) \vspace{6pt} \\
H_3(p_5, p_6, p_7, 0, 0)
\end{pmatrix},
\]
where we use $H_m(h_0, \ldots, h_{2m-1})$ to denote the $m \times m$ Hankel matrix
\[
\begin{pmatrix}
h_0 & h_1 & h_2 & \cdots & h_{m-2} & h_{m-1}  & h_m \\
h_1 & h_2 & h_3 & \cdots & h_{m-1} & h_m 	& h_{m+1} \\
h_2 & h_3 & h_4 & \cdots & h_m		& h_{m+1}  & h_{m+2} \\
\vdots &&& \ddots  &&& \vdots \\
h_{m-2} & h_{m-1} & h_m & \cdots & h_{2m-5} & h_{2m-4} & h_{2m-3}\\
h_{m-1}& h_m & h_{m+1} & \cdots & h_{2m-4} & h_{2m-3} & h_{2m-2} \\
h_m & h_{m+1} &h_{m+2} & \cdots & h_{2m-3}& h_{2m-2} & h_{2m-1} 
\end{pmatrix}.
\]
Now, the quantum moduli space is obtained from this family after dividing by 
the equivalence relation $\sim$ over each point $p$.
Hence 
\[
\mathfrak M^{\hbar}_2(\mathcal W_2, \sigma) = \mathfrak U/ \sim.
\]
%
% foliated by $r$-planes where
%$r$ is the number of linearly independent vectors appearing on the right hand side of the equation. Such $r$-planes
%form the equivalence classes over $p$ an consequently, the moduli space over $p$ is parametrised by 
%an $(8-r)$-plane.

We conclude that $\mathfrak M^{\hbar}_2(\mathcal W_2, \sigma)\rightarrow \mathfrak M_2(W_2)\simeq \mathbb P^7
$ (where the isomorphism is given by Lem.~\ref{classicalmoduli})
is the \'etale space of a constructible sheaf, with stalk at $p$  having  dimension equal to the 
corank of $\mathfrak U_p$, in this case
\[
 3\leq \dim {\mathfrak M^{\hbar}_2(\mathcal W_2, \sigma)}_ p =
 \mbox{corank}\,  \mathfrak U_p= 6-\rk \mathfrak U_p\leq 6.\]
 Now we generalize the result for any $j$. 
 The calculation is made using the same steps as in the case $j=3$ 
by writing down separately equations for each individual monomial $z^lu_1$ or $z^lu_2$ and assuming that all the brackets in the term $\trevo$ vanish. 
 We then check which terms can be cancelled out by an appropriate choice of $\lambda$.

Carrying out the calculations we obtain the linear system
\[
\begin{pmatrix}
q'_{0}- p'_{0} \\
\vdots \\
q'_{4j-5}- p'_{4j-5}\\
\end{pmatrix}=
\begin{pmatrix}
H_{2j-3}(p_0, \ldots, p_{2j-4}, 0, \ldots, 0)  \quad {\bf 0} \, \vspace{6pt}\\
H_{2j-1}(p_{2j-3}, \ldots, p_{4j-5}, 0, \ldots, 0)\\
\end{pmatrix}
\begin{pmatrix}
\lambda_0 \\
\vdots    \\
\lambda_{2j-2} \\
\end{pmatrix},
\]
where the matrix of the coefficients contains two Hankel blocks and a zero matrix with two columns.
Analogously to the case $j_3$, we can always choose $\lambda_{2j-3}$ and $\lambda_{2j-2}$ to solve the equations involving $q'_{2j-3}$ and $q'_{2j-2}$. We have then the simplified system
\[
\begin{pmatrix}
q'_{0}- p'_{0} \\
\vdots \\
q'_{2j-4}- p'_{2j-4}\\
q'_{2j-1}- p'_{2j-1}\\
\vdots \\
q'_{4j-5}- p'_{4j-5}\\
\end{pmatrix}=
\underbrace{
\begin{pmatrix}
H_{2j-3}(p_0, \ldots, p_{2j-4}, 0, \ldots, 0) \vspace{6pt}\\
H_{2j-3}(p_{2j-1}, \ldots, p_{4j-5}, 0, \ldots, 0)\\
\end{pmatrix}
}_{\mathfrak{U}_p}
\begin{pmatrix}
\lambda_0 \\
\vdots    \\
\lambda_{2j-4} \\
\end{pmatrix}.
\] 
We arrive then at the conclusion
\[
 2j-3 \leq \dim {\mathfrak M^{\hbar}_j(\mathcal W_2, \sigma)}_ p =\mbox{corank}\,  \mathfrak U_p=
   2j - \rk \mathfrak U_p\leq 4j-6.
\]
\end{proof}

\appendix
\section{Computations of $H^1$}
\label{H1-Wk-O}
\begin{lemma}\label{van}
$H^1(W_1, \mathcal{O}_{W_1})= H^1(W_2, \mathcal{O}_{W_2})=0.$
\end{lemma}

\begin{proof}
A 1-cocycle $\tau \in \mathcal O(U \cap V)$ may be written in the form
$$\displaystyle
\tau_U = \sum_{l=-\infty}^\infty \sum_{i=0}^\infty \sum_{s=0}^\infty \tau_{lis}z^lu_1^iu_2^s.
$$
Since terms containing only positive powers of $z$ are holomorphic on the $U$-chart 
\[
\tau_U \sim \sum_{l=-\infty}^{-1} \sum_{i=0}^\infty \sum_{s=0}^\infty \tau_{lis}z^lu_1^iu_2^s,
\]
where $\sim$ denotes cohomological equivalence.
Changing to $V$ coordinates we have
\begin{equation}\label{h1}
\tau_V = \sum_{l=-\infty}^{-1} \sum_{i=0}^\infty \sum_{s=0}^\infty \tau_{lis}\xi^{-l+ki+(-k+2)s} v_1^iv_2^s,
\end{equation}
where, for  $k=1,2$  exponents of $\xi$ are  non-negative. Thus, $\tau_V $ is holomorphic on $V$, and $\tau \sim 0$.
% We then conclude that $$H^1(W_1, \mathcal{O})= H^1(W_2, \mathcal{O})=0.$$
\end{proof}

\begin{lemma}\label{muchlarger}
 $H^1(W_3, \mathcal O_{W_3})$ is infinite dimensional over $\mathbb C$.
 \end{lemma}
 
 \begin{proof} As in the proof of Lem.~\ref{van}
 we arrive at the expression (\ref{h1}) for the 1-cocycle $\tau$ on the $V$-chart,
which in the case $k=3$, gives
\[
\tau_V \sim \sum_{l=-\infty}^{-1} \sum_{i=0}^\infty \sum_{s=0}^\infty \tau_{lis}\xi^{-l+3i-s} v_1^iv_2^s.
\]

The terms that are not holomorphic on $V$ are all of those satisfying 
$-l+3i-s<0.$

We conclude that all  terms having $s> 3i-l$, namely all of
$$ \sum_{l=-\infty}^{-1} \sum_{i=0}^\infty \sum_{s= 3i-l+1}^\infty \tau_{lis}z^{l}u_1^iu_2^s$$
 are nontrivial in first cohomology, so that 
 $\displaystyle \dim H^1(W_3,\mathcal O_{W_3})= \infty.$
 \end{proof}

\textbf{Data availability:}
Data sharing is not applicable to this article as no datasets were generated or analyzed during the current study.

\textbf{Conflict of interest:}
On behalf of all authors, the corresponding author states that there is no conflict of interest. 

\paragraph{\bf Acknowledgements}  
 E. Ballico is a member of  GNSAGA of INdAM (Italy). 
 E. Gasparim is a senior associate of the Abdus Salam International 
 Centre for Theoretical Physics, Trieste (Italy).
   F. Rubilar acknowledges support of ANID-FAPESP cooperation 2019/13204-0.
 B. Suzuki  was supported by Grant 2021/11750-7 S\~ao Paulo Research Foundation - FAPESP.\\

\end{document}